%% file: v7_Transfer_Morphism_for_Symplectomorphism.tex
\documentclass[11pt]{amsart}
\usepackage{amsmath, amssymb, amsfonts, amsthm, mathtools, verbatim, dsfont, graphicx, enumerate, pifont,tikz,tikz-cd}

\usepackage{hyperref}

\usepackage[backend=bibtex]{biblatex}

\theoremstyle{theorem}
\newtheorem{theorem}{Theorem}[section]
\newtheorem{lemma}[theorem]{Lemma}
\newtheorem{proposition}[theorem]{Proposition}
\newtheorem{corollary}[theorem]{Corollary}

\theoremstyle{definition}
\newtheorem{definition}[theorem]{Definition}
\newtheorem{remark}[theorem]{Remark}

\newtheorem{example}[theorem]{Example}

\providecommand{\abs}[1]{\left\lvert#1\right\rvert}
\providecommand{\norm}[1]{\left\lVert#1\right\rVert}
\providecommand{\OP}[1]{\operatorname{#1}}

\renewcommand{\hat}{\widehat}

\begin{document}

\author{Igor Uljarevic}
\title{Viterbo's Transfer Morphism for Symplectomorphisms}
\maketitle
\begin{center}\today \end{center}

\begin{abstract}
We construct an analogue of Viterbo’s transfer morphism for Floer homology of an automorphism of a Liouville domain. As an application we prove that the Dehn-Seidel twist along {\em any} Lagrangian sphere in a Liouville domain of dimension $\geqslant 4$ has infinite order in the symplectic mapping class group.
\end{abstract}

\section{Introduction}

In his 1999 paper \cite{MR1726235}, Viterbo constructed a morphism
\begin{equation}\label{eq:viterbotransfer}
SH_\ast(W_2)\to SH_\ast(W_1)
\end{equation}
associated to a codimension-0 embedding $W_1\to W_2$ of a Liouville domain into another Liouville domain that preserves the Liouville form. The map \eqref{eq:viterbotransfer}, called the transfer morphism, fits into the commutative diagram
\begin{center}
\begin{tikzcd}
{SH_\ast(W_2)} \arrow{r} &{SH_\ast(W_1)}\\
{H_{\ast+n}(W_2,\partial W_2)} \arrow{r} \arrow{u} &{H_{\ast+n}(W_1,\partial W_1)}\arrow{u}.
\end{tikzcd}
\end{center}
Here, $2n$ is the dimension of $W_2$ and the map 
\[H_{\ast+n}(W_2,\partial W_2)\to H_{\ast+n}(W_1,\partial W_1)\]
is the composition of the homomorphism induced by the inclusion $(W_2,\partial W_2)\hookrightarrow (W_2, W_2\setminus W_1)$ and the excision isomorphism
\[H_{\ast+n}(W_2,W_2\setminus W_1)\to H_{\ast+n}(W_1,\partial W_1).\]

The aim of this paper is to construct an analogue of Viterbo's transfer morphism for the groups $HF_\ast(\phi, a)$ associated to an exact
symplectomorphism $\phi$ of a Liouville domain and a slope $a$ (see
\cite{Uljarevic}). As an application we construct an asymptotic growth invariant for such symplectomorphisms. As a byproduct we prove that the Dehn-Seidel twist along {\em any} Lagrangian sphere in a Liouville domain of dimension $\geqslant 4$ is of infinite order in the symplectic mapping class group (see Corollary~\ref{cor:tauinf} below).

Let $a\leqslant a'$ be two admissible slopes (i.e. $a, a'$ are not periods of any
Reeb orbit on the boundary of the Liouville domain). There is a well defined continuation map
\[HF_\ast(\phi,a)\to HF_\ast(\phi,a').\]
Given $W_1$ and $W_2$ as above, an exact symplectomorphism $\phi$ of $W_1$ can also be seen as an exact symplectomorphism of $W_2.$ To avoid any ambiguity, we will write $HF_\ast(W_j,\phi,a)$ for the group $HF_\ast(\phi,a)$ when $\phi$ is seen as a symplectomorphism of $W_j,$ $j\in\{1,2\}.$

\begin{theorem}\label{thm:transfermorphism}
Let $W_1,W_2,\phi$ be as above, and let $a, b\in \mathbb{R}^+\cup\{\infty\} $ be positive admissible slopes (with respect to $W_2$ and $W_1,$ respectively ). Assume $a\leqslant b.$ Then, there exists a linear map
\[HF_\ast(W_2,\phi,a)\to HF_\ast(W_1,\phi,b),\]
called the \textbf{transfer morphism}, with the following properties. It coincides with the map \eqref{eq:viterbotransfer} for $a=b=\infty,$ and $\phi$ equal to the identity. Moreover, the diagram
\begin{equation}\label{eq:diagram}
\begin{tikzcd}
HF_\ast(W_2,\phi,a)\arrow{r}\arrow{d}& HF_\ast(W_1,\phi,b)\arrow{d}\\
HF_\ast(W_2,\phi,a')\arrow{r}&HF_\ast(W_1,\phi,b'),
\end{tikzcd}
\end{equation}
consisting of transfer morphisms and continuation maps, commutes for all  admissible slopes $a',b'\in\mathbb{R}^+\cup\{\infty\}$ such that $a\leqslant a',$ $b\leqslant b',$ and $a'\leqslant b'.$
\end{theorem}
\subsection{Applications}
As an application, we construct a numerical invariant 
\[\kappa(W,\phi):=\limsup_{m\to \infty}\frac{\dim HF(W,\phi^m,\varepsilon)}{m}\] 
for an exact symplectomorphism $\phi$ of a Liouville domain $W,$ the so-called iterated ratio (see also Definition~\ref{def:iteratedratio}). It does not depend on the ambient Liouville domain in the following sense.

\begin{theorem}\label{thm:ambient}
Let $W_1$ and $W_2$ be Liouville domains as in Theorem~\ref{thm:transfermorphism} and let $\phi:W_1\to W_1$ be an exact symplectomorphism, then $\kappa(W_1,\phi)=\kappa(W_2,\phi).$
\end{theorem}

The next theorem calculates $\kappa$ for the square of a Dehn-Seidel twist furnished by a Lagrangian sphere in a Liouville domain.

\begin{theorem}\label{thm:taukappa}
Let $W$ be a Liouville domain of dimension $2n\geqslant 4, $ and let $L\subset W$ be a Lagrangian sphere. Then,
\[ \kappa(W,\tau_L^2)=4, \]
where $\tau_L:W\to W$ stands for a Dehn-Seidel twist furnished by $L.$
\end{theorem}

Note that $\kappa(W,\OP{id})=0.$

\begin{corollary}\label{cor:tauinf}
Any Lagrangian sphere $L$ in any Liouville domain $W$ of dimension greater than or equal to 4 gives rise to a Dehn-Seidel twist $\tau_L$ that represents an element of infinite order in the symplectic mapping class group. In other words, $\tau_L^k$ is not symplectically isotopic to the identity relative to the boundary for all $k\in\mathbb{N}.$
\end{corollary}
\begin{remark}
The analogous (to Corollary~\ref{cor:tauinf}) statement for closed symplectic manifolds is false. For instance, the Dehn-Seidel twist furnished by the Lagrangian diagonal in $\mathbb{S}^2\times\mathbb{S}^2$ is symplectically isotopic to the identity. 
\end{remark}

Theorem~\ref{thm:transfermorphism} (applied to $\phi$ equal to the identity) can be used to obtain some information about closed geodesics.

\begin{corollary}\label{cor:geosphere}
Let $(\mathbb{S}^n,g_0)$ be the $n$-dimensional sphere with the standard Riemannian metric, and let $g$ be a Riemannian metric on $\mathbb{S}^n$ such that $g\leqslant g_0.$ Then, there exists a non-constant closed geodesic on $(\mathbb{S}^n,g)$ of length less than or equal to $2\pi.$
\end{corollary} 
We will prove, in fact, a more general statement (see Theorem~\ref{thm:visiblerankcomparison} and Corollary~\ref{cor:geogeneral}).
\subsection{Conventions}

Let $(W,\lambda)$ be a symplectic manifold. A function $H:W\to\mathbb{R}$ and its Hamiltonian vector field $X_{H}$ are related by $dH=\omega(X_H,\cdot).$ The Hamiltonian isotopy $\psi_t^H:W\to W$ of a time-dependent Hamiltonian $H:\mathbb{R}\times W\to\mathbb{R}:(t,x)\mapsto H_t(x)$ is determined by $\partial_t \psi_t^H=X_{H_t}\circ\psi^H_t,\:\psi_0^H=\OP{id}.$ An almost complex structures $J$ on $W$ is said to be $\omega$-compatible if $\omega(\cdot,J\cdot)$ is a Riemannian metric.

\subsection*{Acknowledgments}
I would like to thank Paul Biran and Dietmar Salamon for many helpful discussions.
\section{Preliminaries}

\subsection{Liouville domains}

\begin{definition}\label{def:liouvilledomain}
A Lioville domain is a compact manifold $W$ with a 1-form $\lambda,$ called Liouville form, that satisfies the following conditions. The 2-form $d\lambda$ is a symplectic form on $W,$ and the Liouville vector field $X_\lambda$, defined by $X_\lambda\lrcorner d\lambda=\lambda,$ points transversally out on the boundary $\partial W.$
\end{definition}

If $(W,\lambda)$ is a Liouville domain, then the restriction $\left.\lambda\right|_{\partial W}$ of $\lambda$ to the boundary $\partial W$ is a contact form. Moreover, a collar neighbourhood of $\partial W$ can be identified with $(\partial W\times (0,1], r\left.\lambda\right|_{\partial W})$ using the flow of the Liouville vector field $X_\lambda.$ Here, $r$ stands for the real coordinate.

\begin{definition}
Let $(W,\lambda)$ be a Liouville domain. There exists a unique embedding
\[\iota :\partial W\times (0,1]\to W\]
such that $\iota(x,1)=x$ and $\iota^\ast \lambda=r \left.\lambda\right|_{\partial W}.$ The \textbf{completion} $\hat{W}$ of $(W,\lambda)$ is the exact symplectic manifold obtained by gluing $W$ and $\partial W\times (0,\infty)$ via $\iota.$
\end{definition}

If $W$ is a Liouville domain and $r\in(0,\infty)$, we denote by $W^r$ the subset of the completion $\hat{W}$ defined by
\[W^r:=\left\{\begin{matrix} W\setminus (\partial W\times (r,1])&\text{for }r<1\\ W\cup \partial W\times(1,r]&\text{for }r\geqslant 1 \end{matrix} \right..\]
Here, and in the rest of the paper, the sets $W$ and $\partial W\times(0,\infty)$ are identified with the  corresponding regions in the completion $\hat{W}.$ 

\subsection{Floer homology for exact symplectomorphisms}

\begin{definition}
An \textbf{exact symplectomorphism} $\phi$ of a Liouville domain $W$ with connected boundary is a diffeomorphism of $W$ equal to the identity near the boundary such that the 1-form $\phi^\ast\lambda-\lambda$ is exact. We denote by $F_\phi:W\to\mathbb{R}$ the unique function that is equal to 0 on the boundary and satisfy $dF_\phi=\phi^\ast\lambda-\lambda.$
\end{definition}

\begin{definition}
Let $\phi$ be an exact symplectomorphism of a Lioville domain $(W,\lambda),$ and let $a$ be an admissible slope. 
\begin{enumerate}
\item Floer data for $(\phi,a)$ consists of a Hamiltonian $H_t:\hat{W}\to \mathbb{R}$ and a family $J_t$ of a $d\lambda$-compatible almost complex structures on $\hat{W}$ satisfying the following conditions
\begin{align*}
&H_{t+1}=H_t\circ\phi,\\
&J_{t+1}=\phi^\ast J_t,
\end{align*}
and such that
\begin{align*}
& H_t(x,r)=ar,\\
& J_t(x,r)\xi=\xi,\\
&J_t(x,r)\partial_r=R,
\end{align*}
for $(x,r)\in\partial W\times (r_0,\infty)$ and $r_0\in(0,\infty)$ large enough. Here, $\xi$ and $R$ stand for the contact structure and the Reeb vector field on $\partial W,$ respectively.
\item Floer data $(H,J)$ for $(\phi,a)$ is said to be regular if
\[\det (d(\phi\circ\psi^H_1)(x)-\OP{id})\not=0\]
for all fixed points $x$ of $\phi\circ\psi_1^H,$ and if the linearized operator of the Floer equation \eqref{eq:Floer} below is surjective.
\end{enumerate}
\end{definition}

Let $\phi:W\to W$ be an exact symplectomorphism, let $a$ be an admissible slope, and let $(H,J)$ be regular Floer data for $(\phi,a).$ The Floer homology $HF_\ast(W,\phi,H,J)$ is the Morse homology for the action functional 
\[\mathcal{A}_{\phi,H}:\Omega_\phi\to\mathbb{R},\]
defined on the space of ($\phi$-)twisted loops
\[\Omega_\phi:=\left\{ \gamma:\mathbb{R}\to\hat{W}\::\:\phi(\gamma(t+1))=\gamma(t)\right\}\]
by
\[\mathcal{A}_{\phi,H}(\gamma):= -\int_0^1 \left(\gamma^\ast\lambda + H_t(\gamma(t))dt\right) - F_\phi(\gamma(1)).\]
In particular, the chain complex, denoted by $CF_\ast(W,\phi,H,J),$ is generated by critical points of $\mathcal{A}_{\phi,H},$ which happen to be the Hamiltonian twisted orbits
\[\OP{Crit}\mathcal{A}_{\phi,H}= \left\{\gamma\in\Omega_\phi\::\:\dot{\gamma}(t)=X_{H_t}(\gamma(t))\right\},\]
and the differential is obtained by counting unparametrized isolated solutions 
\[u:\mathbb{R}^2\to\hat{W}\]
of the Floer equation
\begin{equation}\label{eq:Floer}
\partial_s u+J_t(u)(\partial_t u-X_{H_t}(u))=0
\end{equation}
with periodicity condition
\[\phi(u(s,t+1))=u(s,t).\]
Different choices of regular Floer data for $(\phi,a)$ lead to canonically isomorphic Floer homologies. Hence the group $HF_\ast(W,\phi,a)$ is well defined. The group $HF_\ast(W,\OP{id},a)$ will also be denoted by $HF_\ast(W,a).$

In the next two lemmas, we prove that the Floer homology $HF_\ast(W,\phi,a)$ remains the same if we change the Liouville form by adding an exact 1-form.

\begin{lemma}
Let $(W,\lambda)$ be a Liouville domain and let $f:W\to \mathbb{R}$ be a function. Then, a map $\phi: W\to W$ is an exact symplectomorphism of $(W,\lambda)$ if, and only if, it is an exact symplectomorphism of $(W,\lambda +df).$ 
\end{lemma}
\begin{proof}
Denote the 1-form $\lambda+df$ by $\widetilde{\lambda}.$ Assume $\phi$ is an exact symplectomorphism of $(W,\lambda).$ Then, there exists a function $F:W\to\mathbb{R}$ such that
\[ \phi^\ast\lambda-\lambda= dF, \]
and $\phi$ is compactly supported. Hence
\[
\begin{split}
\phi^\ast\widetilde{\lambda}- \widetilde{\lambda} &= \phi^\ast(\lambda +df)-\lambda-df\\
&= \phi^\ast\lambda-\lambda + d(f\circ \phi - f)\\
&= d(F+ f\circ \phi- f),
\end{split} 
\]
i.e. $\phi$ is an exact symplectomorphism of $(W,\widetilde{\lambda}).$ The opposite direction is proven analogously.
\end{proof}

\begin{lemma}\label{lem:changinglambda}
Let $(W,\lambda)$ be a Liouville domain, let $\phi: W\to W$ be an exact symplectomorphism, and let $f:W\to\mathbb{R}$ be a function equal to $0$ near the boundary. Then,
\[ HF_\ast(W,\lambda,\phi,a)\cong HF_\ast(W,\lambda+ df, \phi, a), \]
for all admissible slopes $a.$
\end{lemma}
\begin{proof}
Denote by $\widetilde{\lambda}$ the 1-form $\lambda+ df.$ Since $\lambda$ and $\widetilde{\lambda}$ agree near the boundary (where $f$ is equal to 0), the slope of a Hamiltonian with respect to $(W,\lambda)$ is the same as the one with respect to $(W,\widetilde{\lambda}).$ Moreover, the completions of $(W,\lambda)$ and $(W,\widetilde{\lambda})$   are symplectomorphic. Let $H$ be a Hamiltonian with the slope equal to $a.$ Denote by $\widetilde{F}_\phi$ the compactly supported function $\hat{W}\to\mathbb{R}$ such that
\[ \phi^\ast\widetilde{\lambda} - \widetilde{\lambda}= d\widetilde{F}_\phi. \]
The functions $\widetilde{F}_\phi$ and $F_\phi$ are related by
\[ \widetilde{F}_\phi= F_\phi+ f\circ \phi - f. \]
The lemma follows from the following sequence of equalities
\[
\begin{split}
\widetilde{\mathcal{A}}_{\phi,H} (\gamma):=& -\int_0^1\left( \gamma^\ast\widetilde{\lambda} + H_t(\gamma(t))dt \right) - \widetilde{F}_\phi(\gamma(1))\\
=& \mathcal{A}_{\phi,H}(\gamma) - \int_0^1 \frac{d}{dt} (f\circ\gamma) dt - f\circ\phi(\gamma(1))+ f(\gamma(1))\\
=& \mathcal{A}_{\phi,H}(\gamma) - f(\gamma(1))+ f(\gamma(0)) - f\circ\phi(\gamma(1))+ f(\gamma(1))\\
=& \mathcal{A}_{\phi,H}(\gamma),
\end{split} 
\]
for $\gamma\in\Omega_\phi.$
\end{proof}
The graded group $CF^{<c}_\ast(W,\phi,H,J),$ $c\in\mathbb{R},$ generated by the elements of $\OP{Crit}\mathcal{A}_{\phi,H}$ having the action less than $c$ is a subcomplex of $CF_\ast(W,\phi,H,J).$ It fits into a short exact sequence of chain complexes
\[0\to CF_\ast^{<c}(W,\phi,H,J)\to CF_\ast(W,\phi,H,J)\to CF_\ast^{\geqslant c}(W,\phi,H,J)\to 0,\]
where
\[CF_\ast^{\geqslant c}(W,\phi,H,J):=\frac{CF_\ast(W,\phi,H,J)}{CF_\ast^{<c}(W,\phi,H,J)}.\]
The homologies of the chain complexes $CF_\ast^{< c}(\cdots)$ and $CF_\ast^{\geqslant c}(\cdots)$ are denoted by $HF_\ast^{< c}(\cdots)$ and $HF_\ast^{\geqslant c}(\cdots),$ respectively.

For more details on Floer homology for an exact symplectomorphism, see \cite{Uljarevic}.

\subsection{Notation}\label{sec:notation}
Throughout, $(W_2,\lambda)$ is a $2n$-dimensional Liouville domain with connected boundary, $W_1\subset W_2$ is a codimension-0 submanifold with connected boundary such that $(W_1,\lambda)$ is a Liouville domain in its own right, and $\phi:W_1\to W_1$ is an exact symplectomorphism.

\section{Transfer morphism}

The transfer morphism is essentially the map
\[HF_\ast(W_2,\phi, H,J)\to HF_\ast^{\geqslant 0}(W_2,\phi,H,J)\]
induced by the natural projection of chain complexes $CF_\ast(\cdots)\to CF_\ast^{\geqslant 0}(\cdots)$ for a Hamiltonian $H$ that is $C^2$-close to a so-called stair-like Hamiltonian. The group $HF_\ast(W_2,\phi, H,J)$ is isomorphic to $HF_\ast(W_2,\phi, a),$ where $a$ is the slope of $H,$ and $HF_\ast^{\geqslant 0}(W_2,\phi,H,J)$ can be identified with $HF_\ast(W_1,\phi,b)$ for a certain slope $b.$ In the rest of the section, we describe the construction in more details and prove Theorem~\ref{thm:transfermorphism}.

\subsection{Stair-like Hamiltonians}

\begin{definition}\label{def:Hstair}
Let $0<a<b$ be such that $a$ is admissible with respect to $W_2$ and $b$ is admissible with respect to $W_1,$ and let $b_0$ be the greatest period of some Reeb orbit on $\partial W_1$ that is smaller than $b.$ The set $\mathcal{H}_{\OP{stair}}(\phi, W_1, W_2, a, b)$ is defined to be the set of (time-independent) Hamiltonians $H:\hat{W}_2\to\mathbb{R}$ having the following property. There exist positive real numbers $\delta_1,\delta_2,\delta_3, A, B, C\in(0,\infty)$ and functions
\begin{align*}
h_1&:[\delta_1,2\delta_1]\to \mathbb{R},\\
h_2&:[1-\delta_2,1]\to\mathbb{R},\\
h_3&: [1+\delta_3,1+2\delta_3]\to\mathbb{R},
\end{align*}
such that
\begin{enumerate}
\item $\phi$ is compactly supported in the interior of $W^{\delta_1},$
\item $\sup\limits_{p\in W_1}\abs{F_\phi(p)}<A,$
\item $h_1$ and $h_3$ are convex and strictly increasing,
\item $h_2$ is concave, strictly increasing, and $h_2(r)>rb_0$ for all $r\in [1-\delta_2,1]$
\end{enumerate}
and such that
\[H(p)=\left\{ 
\begin{matrix}
-A & \text{for } p\in W_1^\delta,\\
h_1(r) & \text{for } p=(x,r)\in \partial W_1\times[\delta_1,2\delta_1),\\
b(r-2\delta_1) & \text{for } p=(x,r)\in\partial W_1\times[2\delta_1,1-\delta_2),\\
h_2(r) & \text{for } p=(x,r)\in\partial W_1\times [1-\delta_2, 1],\\
B & \text{for } p\in W_2^{1+\delta_3}\setminus W_1,\\
h_3(r) & \text{for } p=(x,r)\in\partial W_2\times(1+\delta_3,1+2\delta_3),\\
ar+C & \text{for } p=(x,r)\in\partial W_2\times[1+2\delta_3,\infty).
\end{matrix}
\right.\]
\end{definition}
A typical element of $\mathcal{H}_{\OP{stair}}(\phi, W_1, W_2, a, b)$ is shown in Figure~\ref{fig:stair-like}.
\begin{figure}
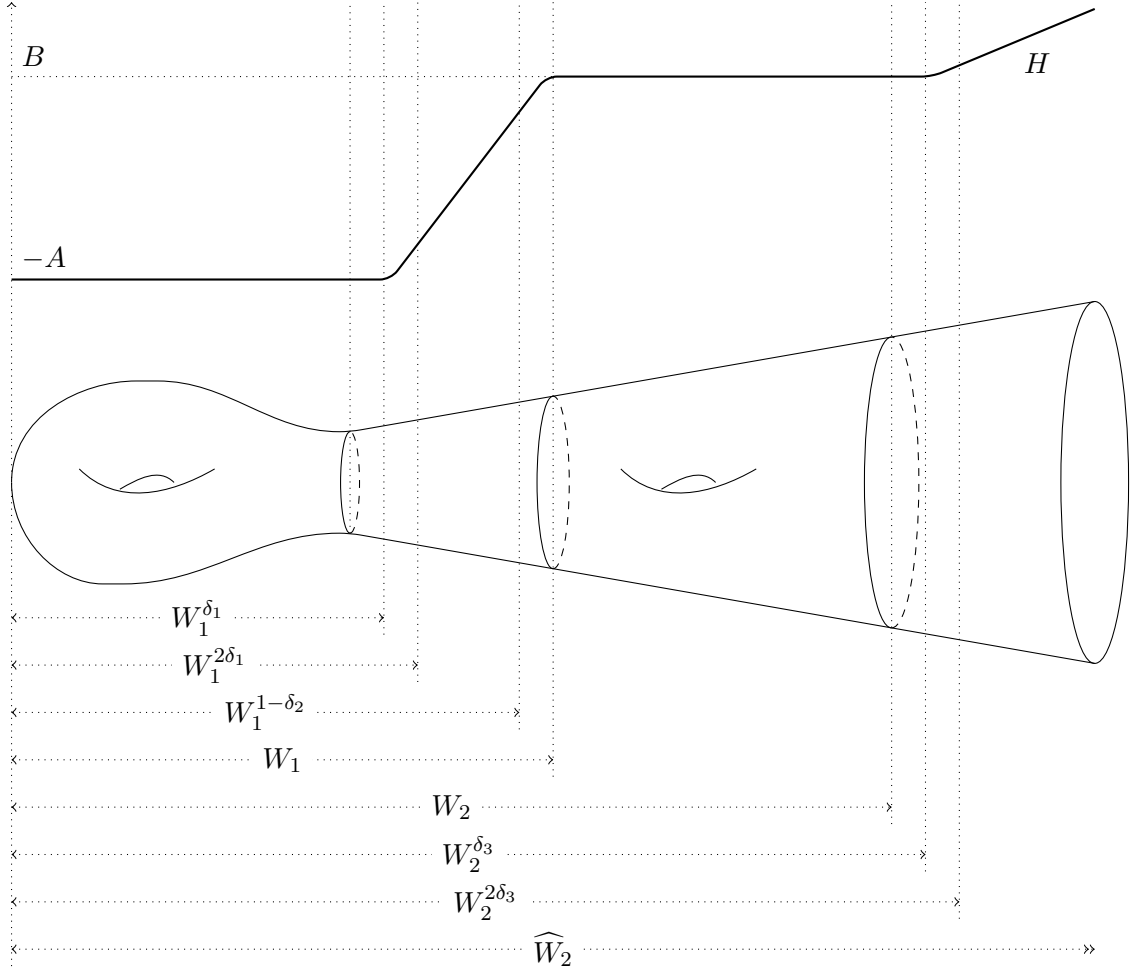

\include{transferfigure}
\caption{A stair-like Hamiltonian}
\label{fig:stair-like}
\end{figure}

\begin{definition}
Let $W$ be a Liouville domain, and let $\phi:\hat{W}\to\hat{W}$ be a diffeomorphism. The \textbf{support radius} $\rho(W,\phi)$ of $\phi$ is defined by
\[\rho(W,\phi):=\inf \left\{ r\in(0,\infty)\::\:\OP{supp} \phi\subset W^r\right\}.\]
\end{definition}

\begin{lemma}
Let $a,b,$ and $b_0$ be as in Definition~\ref{def:Hstair}. If 
\[C(W_1,\phi)< \min\left\{b-b_0, b-a\right\},\]
where
\begin{equation}\label{eq:constant}
C(W_1,\phi):=2\max\left\{\sup_{p\in W_1}\abs{F_\phi(p)}, \rho (W_1,\phi) \right\},
\end{equation}
then the set $\mathcal{H}_{\OP{stair}}(\phi,W_1,W_2,a,b)$ is not empty. 
\end{lemma}
\begin{proof}
The constants $\delta_1,\delta_2,\delta_2, A, B, C$ from Definition~\ref{def:Hstair} are chosen in the following way and the following order.
\begin{itemize}
\item Choose
\[ A\in \left( \frac{C(W_1,\phi)}{2} , \frac{\min\{b-b_0,b-a \}}{2}  \right), \] 
\item choose $\delta_1\in\left(\frac{A}{b},\frac{\min\{b-b_0,b-a\}}{2b}\right), $
\item choose $\delta_2\in \left(0, 1-\frac{2b\delta_1}{b-b_0}\right),$
\item choose $B\in(\max\{a, b_0, b-2b\delta_1-b\delta_2\}, b-2b\delta_1),$
\item choose $\delta_3\in\left(0,\frac{B-a}{a}\right),$
\item choose $C\in(\max\{0, B-a(1+2\delta_3)\}, B-a(1+\delta_3)).$
\end{itemize}
It follows that the following inequalities hold
\begin{align}
& \label{eq:ineq1} 0<-A-(2b\delta_1)-b\delta_1<b\delta_1,\\
& \label{eq:ineq2} 0<B-(-2b\delta_1)-b(1-\delta_2)<b\delta_2,\\
& \label{eq:ineq3} b(1-\delta_2)+(-2b\delta_1)>(1-\delta_2)b_0,\\
& \label{eq:ineq4} B>b_0,\\
& \label{eq:ineq5} 0<B-C-a(1+\delta_3)<a\delta_3.
\end{align}
Due to Lemma~\ref{lem:tehlem} and Lemma~\ref{lem:tehlem2}, \eqref{eq:ineq1} implies that $h_1$ exists, \eqref{eq:ineq2}, \eqref{eq:ineq3}, \eqref{eq:ineq4} imply that $h_2$ exists, and \eqref{eq:ineq5} implies that $h_3$ exists. This finishes the proof.
\end{proof}

\begin{lemma}
Let $\phi, a, b, H, \delta_1, \delta_2, \delta_3, A, B, C$ be as in Definition~\ref{def:Hstair}. The ranges of the action functional $\mathcal{A}_{\phi,H}$ when evaluated on $\phi$-twisted Hamiltonian orbits contained in the regions
\[\begin{split} 
&I:=W_1^{\delta_1},\quad II:=\partial W_1\times(\delta_1,2\delta_1),\quad III:=\partial W_1\times(1-\delta_2,1),\\
&IV:= W_2^{1+\delta_3}\setminus \OP{int} W_1,\quad V:=\partial W_2\times(1+\delta_3,1+2\delta_3)
\end{split}\]
are given in the following table. \\
\begin{center}
\begin{tabular}{l|c|c|c|c|c}
region& I&II&III&IV&V\\\hline
$\mathcal{A}_{\phi,H} \in $ & $[A-\norm{F_\phi}, A+\norm{F_\phi}]$ & $(A, 2b\delta_1)$ & $(-B,0)$ & $\{-B\}$ & $(-B,-C)$
\end{tabular}
\end{center}\vspace{0.5cm}
Note that each $\phi$-twisted orbit of $H$ is contained in one of these regions.
\end{lemma}
\begin{proof}
We will prove the statement only for region III. For the other regions, the proof is similar and even more direct. 

The symplectomorphism $\phi$ is equal to the identity in region III. Therefore, the twisted Hamiltonian orbits coincide with Hamiltonian loops. Additionally, due to the form of the Hamiltonian $H$ in this region, they can be explicitly described in terms of periodic Reeb orbits on $\partial W_1.$ Each 1-periodic Hamiltonian orbit in region III is given by
\begin{equation}\label{eq:hamreeb}
t\mapsto \left(\gamma(-h_2'(r_0)t), r_0\right)
\end{equation}
where $\gamma:\mathbb{R}\to \partial W_1$ is a $h'_2(r_0)$-periodic Reeb orbit. The action of \eqref{eq:hamreeb} is equal to
\[r_0h'_2(r_0)-h_2(r_0).\]
Consider the function
\[f:[1-\delta_2,1]\to\mathbb{R}\::\: r\mapsto rh'_2(r)-h_2(r).\]
Since $f'(r)=rh''_2(r)\leqslant 0,$ $f$ is decreasing. Hence the action of a 1-periodic Hamiltonian orbit in region III lies in the interval $[f(1),f(r_1)]=(-B,f(r_1)),$ where $r_1\in(1-\delta_2,1)$ is the smallest number such that $h'_2(r_1)$ is a period of some Reeb orbit on $\partial W_1.$ Since $h_2'$ is continuous and decreasing ($h_2$ is concave), we get $h'_2(r_1)=b_0.$ Therefore,
\[f(r_1)=r_1b_0-h_2(r_1)<0\]
(because, by definition, $h_2(r)>rb_0$ for all $r\in[1-\delta_2,1]$). This finishes the proof.
\end{proof}

\begin{proposition}\label{pro:shrinking}
Let $\phi:W_1\to W_1$ be an exact symplectomorphism and let $\phi_t:W_1\to W_1$ be the isotopy of exact symplectomorphisms given by
\[\phi_t:= \left(\psi_t^\lambda\right)^{-1}\circ\phi\circ\psi^\lambda_t,\]
where $\psi_t^\lambda:\hat{W}_1\to\hat{W}_1$ is the Liouville flow. Then,
\[F_{\phi_t}= e^{-t} F_\phi\circ \psi_t^\lambda\quad \text{and}\quad \rho(\phi_t,W_1,\lambda)=e^{-t}\rho(\phi, W_1,\lambda).\]
Consequently, $C(W_1,\phi_t)=e^{-t}C(W_1,\phi).$
\end{proposition}
\begin{proof}
Denote $\psi^\lambda_t$ by $\psi_t$ for simplicity. The Cartan formula implies
\[\psi^\ast_t\lambda=e^t\lambda\quad \text{and}\quad \left(\psi^{-1}_t\right)^\ast\lambda=e^{-t}\lambda.\]
Therefore
\[\begin{split}
dF_{\phi_t}&=\phi^\ast_t\lambda-\lambda= \left(\psi^{-1}_t\circ\phi\circ\psi_t \right)^\ast\lambda-\lambda =\psi^\ast_t\phi^\ast \left(\psi_t^{-1}\right)^\ast\lambda-\lambda\\
&=e^{-t}\psi_t^\ast\left(\phi^\ast \lambda-\lambda\right) = e^{-t}d(F_\phi \circ\psi_t).
\end{split}\] 
Hence $F_{\phi_t}=e^{-t} F_\phi\circ\psi_t.$ Since
\[\OP{supp} \phi_t\subset W^r\quad\Longleftrightarrow\quad \OP{supp}\phi \subset W^{e^tr},\]
the equality $\rho(W_1,\phi_t)=e^{-t}\rho (W_1,\phi)$ holds.
\end{proof}

\subsection{Construction}\label{sec:construction}

\begin{definition}
Let $a$ and $b$ be as in Definition~\ref{def:Hstair}. \textbf{Transfer data} for $(\phi,W_1,W_2,a,b)$ is Floer data $(H,J)$ for $(\phi: W_2\to W_2, b)$ satisfying the following. There exists a Hamiltonian $\overline{H}\in\mathcal{H}_{\OP{stair}}(\phi,W_1,W_2,a,b)$ such that $H=\overline{H}$ on $\partial W_1\times [2\delta_1,1-\delta_2]$ and $\partial W_2\times [1+2\delta_3,\infty],$ and such that the action functional is positive at twisted Hamiltonian orbits in $W_1^{2\delta_1}$  and negative at the other twisted orbits (this holds for $H$ that is $C^2$-close to $\overline{H}$). Additionally, 
\[ dr\circ J_t= -\lambda \]
in $\partial W_1\times[2\delta_1,1-\delta_2].$
\end{definition}

\begin{lemma}
Let $(H,J)$ be a regular transfer data for $(\phi,W_1,W_2,a,b).$ By definition, there exist $\delta_1,\delta_2\in(0,1)$ such that $2\delta_1+\delta_2<1$ and $H_t(x,r)=b(r-2\delta_1)$ for all $(x,r)\in \partial W_1\times[2\delta_1,1-\delta_2).$ Let $G_t:\hat{W}_1\to\mathbb{R}$ be the Hamiltonian defined by
\[G_t(p):=\left\{ \begin{matrix} H_t(p)&\text{for }p\in W_1^{1-\delta_2}\\ b(r-2\delta_1)&\text{for }p\in \partial W_1\times[2\delta_1,\infty). \end{matrix}\right.\]
Then,
\[HF_\ast(W_1,\phi,G,J)=HF_\ast^{\geqslant 0}(W_2,\phi,H,J).\]
\end{lemma}
\begin{proof}
The generators of the chain complexes $CF_\ast(W_1,\phi,G,J)$ and $CF_\ast^{\geqslant 0}(W_2,\phi,H,J)$ coincide. Therefore it is enough to prove the following. The solutions
\[u:\mathbb{R}^2\to\hat{W}_2,\quad \phi(u(s,t+1))=u(s,t)\] 
of the Floer equation which connect two Hamiltonian twisted orbits in $W_1$ satisfy $u\left(\mathbb{R}^2\right)\subset W_1.$ This follows from the part a) of Theorem~4.5 in \cite{Gutt2015}. See also \cite[Proposition~4.4]{Gutt2015}.
\end{proof}

We are going to define the transfer morphism in couple of steps. Definition~\ref{def:TMstep1} defines it for positive slopes with additional technical conditions. Definition~\ref{def:TMstep2} and Definition~\ref{def:TMstep3} eliminate the technical conditions. And, finally, Proposition~\ref{prop:TMcont} extends the transfer morphism to the case of infinite slopes.

\begin{definition}\label{def:TMstep1}
Let $a,b\in(0,\infty)$ be positive real numbers such that $a$ and $b$ are admissible with respect to  $W_2$ and $W_1,$ respectively, and such that $a<b.$ Let $b_0$ be the greatest period of some Reeb orbit on $\partial W_1$ that is smaller than $b.$ Assume
\[C(W_1,\phi)< \min\left\{b-b_0, b-a\right\}.\]
The \textbf{transfer morphism} is defined to be the map 
\[HF_\ast(W_2,\phi,a)\to HF_\ast(W_1,\phi,b),\]
induced by the natural map of chain complexes
\[CF_\ast(W_2,\phi, H,J)\to CF_\ast^{\geqslant 0}(W_2, \phi,H,J),\]
where $(H,J)$ is regular Transfer data for $(\phi,W_1,W_2,a,b).$
\end{definition}

\begin{definition}\label{def:TMstep2}
Let $a\in(0,\infty)$ be admissible with respect to $W_2,$ and let $b\in (0,\infty)$ be admissible with respect to $W_1.$ Assume $a<b.$ The transfer morphism is defined as the composition
\[
\begin{tikzcd}[row sep= large]
HF_\ast(W_2,\phi,a)\arrow[dashed]{r}\arrow{d}[swap]{I(\{\phi_t\})}& HF_\ast(W_1,\phi, b)\\
HF_\ast(W_2,\phi_1,a)\arrow{r}{T}& HF_\ast(W_1,\phi_1,b).\arrow{u}[swap]{I(\{\phi_t\})^{-1}}
\end{tikzcd} 
 \]
Here, 
\[ \phi_t:= (\psi_{ct}^\lambda)^{-1}\circ \phi\circ \psi_{ct}^\lambda ,\qquad t\in[0,1],\]
for $c\in(0,\infty)$ such that 
\[ C(W_1, \phi_1)<\min\{b-b_0, b-a\}, \]
$T$ is the transfer morphism of Definition~\ref{def:TMstep1}, and $I(\{\phi_t\})$ is the isomorphism furnished by isotopy $\{\phi_t\}$ (see Section~2.8 in~\cite{Uljarevic}). 
\end{definition}

\begin{lemma}
In the situation of Definition~\ref{def:TMstep2}, the transfer morphism does not depend on the choice of $c.$
\end{lemma}
\begin{proof}
Let $c_1$ and $c_2$ be positive real numbers such that
\[ C(W_1, \phi_i)<\min\{b-b_0, b-a\},\qquad i\in\{1,2\}, \]
where 
\[ \phi_i:= (\psi_{c_i}^\lambda)^{-1}\circ \phi\circ \psi_{c_i}^\lambda,\qquad i\in\{1,2\}. \]
It is enough to prove that the diagram 
\[\begin{tikzcd}
HF_\ast(W_2,\phi_1,a)\arrow{r}& HF_\ast(W_1,\phi_1, b)\arrow{d}\\
HF_\ast(W_2,\phi,a)\arrow{u}\arrow{d}& HF_\ast(W_1,\phi, b)\\
HF_\ast(W_2,\phi_2, a)\arrow{r}& HF_\ast(W_1,\phi_2, b)\arrow{u}   
\end{tikzcd}\]
commutes, where the vertical arrows stand for the isomorphisms induced by symplectic isotopies. These isomorphisms, as it follows from their definition in \cite[Section~2.8]{Uljarevic}, coincide with naturality isomorphisms (see Section~2.7 in \cite{Uljarevic}) with respect to some Hamiltonians that are compactly supported in the interior of $W_1.$ Hence, it is enough to prove that the diagram
\[\begin{tikzcd} 
HF(W_2,\phi_1, a)\arrow{r}\arrow{d}[swap]{\mathcal{N}(H)} & HF(W_1,\phi_1, b)\arrow{d}{\mathcal{N}(K)}\\
HF(W_2,\phi_2,a)\arrow{r}& HF(W_1,\phi_2,b) 
\end{tikzcd}\]
commutes, where $H,K:\hat{W}_2\to \mathbb{R}$ are compactly supported in the interior of $W_1,$  and $\mathcal{N}(H), \mathcal{N}(K)$ denote the corresponding naturality isomorphisms. This follows from the following fact.  The naturality with respect to a Hamiltonian $\hat{W}_2\to\mathbb{R}$ that is compactly supported in the interior of $W_1$ converts transfer data for $(\phi_1, W_1, W_2,a,b)$ into transfer data for $(\phi_2, W_1, W_2, a, b).$ 
\end{proof}

\begin{proposition}\label{prop:TMcont}
The diagram
\[ \begin{tikzcd}
HF_\ast(W_2,\phi,a)\arrow{r}\arrow{d}& HF_\ast(W_1,\phi,b)\arrow{d}\\
HF_\ast(W_2,\phi,a')\arrow{r}&HF_\ast(W_1,\phi,b'),
\end{tikzcd} \]
consisting of transfer morphisms and continuation maps, commutes whenever $a,b,a',b'\in(0,\infty)$ are such that all the maps in the diagram are well defined.
\end{proposition}

\begin{proof}[Proof of Theorem~\ref{thm:transfermorphism}]
The proof follows from the proof of Proposition~4.7 in \cite{Gutt2015}.
\end{proof}

\begin{definition}\label{def:TMstep3}
Let $a\in(0,\infty)$ be admissible with respect to both $W_1$ and $W_2.$ Let $\varepsilon>0$ be a positive number such that the elements of $[a, a+\varepsilon]$ are all admissible with respect to $W_1.$ The transfer morphism
\[ HF_\ast(W_2,\phi,a)\to HF_\ast(W_1,\phi,a) \]
is defined to be the composition of the transfer morphism
\[ HF_\ast(W_2,\phi,a)\to HF_\ast(W_1,\phi, a+\varepsilon) \]
of Definition~\ref{def:TMstep2} and the inverse of the continuation map
\[ \Phi^{-1}\::\: HF_\ast(W_1,\phi,a+\varepsilon)\to HF_\ast (W_1,\phi,a). \]
\end{definition}

\begin{lemma}
The transfer morphism of Definition~\ref{def:TMstep3} is well defined and does not depend on the choice of $\varepsilon.$ 
\end{lemma}

\begin{proof}
Lemma~\ref{lem:periodsclosed} implies that $\varepsilon$ from Definition~\ref{def:TMstep3} exists. Due to \cite[Lemma~2.25]{Uljarevic}, the continuation map
\[ \Phi \: :\: HF_\ast(W_1,\phi,a)\to HF_\ast (W_1,\phi, a+\varepsilon) \]
is an isomorphism. Hence, its inverse $\Phi^{-1}$ is well defined. Independence of the choice of $\varepsilon$ follows from Proposition~\ref{prop:TMcont}.
\end{proof}

Proposition~\ref{prop:TMcont} enables us to extend the transfer morphism to the case of infinite slopes. The proposition is still true after the extensions of the transfer morphism.

\begin{proof}[Proof of Theorem~\ref{thm:transfermorphism}]
The existence of the transfer morphism has already been shown above. It remains to prove that the diagram \eqref{eq:diagram} commutes. This follows from the proof of Proposition~4.7 in \cite{Gutt2015}.
\end{proof}

\section{Applications}

\subsection{Iterated ratio}

\begin{definition}\label{def:iteratedratio}
Let $\phi: W\to W$ be an exact symplectomorphism of a Liouville domain $W.$ The \textbf{iterated ratio} $\kappa(W,\phi)$ is defined to be the number
\[\kappa(W,\phi):=\limsup_{m\to \infty}\frac{\dim HF(W,\phi^m,\varepsilon)}{m}.\]
Here $\varepsilon>0$ is small enough (smaller than any positive period of some Reeb orbit on $\partial W$). 
\end{definition}

\begin{remark}
The iterated ratio is invariant under compactly supported symplectic isotopies. It measures the linear growth rate of $\dim HF(W,\phi^m,\varepsilon).$ There is yet another similar invariant, the Floer theoretic entropy,
\[h_{Floer}(W,\phi):=\limsup_{m\to\infty} \frac{\log\left(\dim HF(W,\phi^m,\varepsilon)\right)}{m},\]
which measures the exponential growth rate of $\dim HF(W,\phi,\varepsilon).$ If $W$ is a surface and $\phi$ its area-preserving diffeomorphism that has a pseudo-Anosov component, then $h_{Floer}(W,\phi)>0$ \cite[page 167]{MR2929086}. Consequently, $\kappa(W,\phi)$ is equal to infinity.
\end{remark}

\begin{proof}[Proof of Theorem~\ref{thm:ambient}]
Let $\varepsilon>0$ be small. It is enough to prove that the number
\[\abs{\dim HF(W_1,\phi^m,\varepsilon)-\dim HF(W_2,\phi^m,\varepsilon)}\]
is bounded (by a constant not depending on $m$). Let $\psi_m:W_1\to W_1$ be an exact symplectomorphism that is isotopic to $\phi^m$ through exact symplectomorphisms, and such that $C(W_1,\psi_m)<\frac{\varepsilon}{2}.$ Such a symplectomorphism exists for each $m\in\mathbb{N}$ due to Proposition~\ref{pro:shrinking}. Choose transfer data $(H^m,J^m)$ for $(\psi_m,W_1,W_2,\frac{\varepsilon}{2},\varepsilon),$ such that $H^m=H^1$ and $J^m=J^1$ on $W_2\setminus W_1.$ The short exact sequence
\[0\to CF_\ast^{<0}(W_2,\psi_m,H^m, J^m)\to CF_\ast(W_2,\psi_m,H^m, J^m)\to CF_\ast^{\geqslant 0}(W_2,\psi_m,H^m, J^m)\to 0\]
of chain complexes induces the long exact sequence in homology
\[\cdots\to HF_\ast^{<0}(W_2,\psi_m,H^m, J^m)\to HF_\ast(W_2,\psi_m,H^m, J^m)\to HF_\ast^{\geqslant 0}(W_2,\psi_m,H^m, J^m)\to \cdots\]
By applying identifications as in Section~\ref{sec:construction}, the long exact sequence is transformed into
\begin{equation}\label{eq:lsequence}
\cdots\to HF_\ast^{<0}(W_2,\psi_m,H^m, J^m)\to HF_\ast\left(W_2,\psi_m,\frac{\varepsilon}{2}\right)\to HF_\ast(W_1,\psi_m,\varepsilon)\to \cdots
\end{equation}
The chain complexes $CF_\ast^{<0}(W_2,\psi_m, H^m,J^m), m\in\mathbb{N}$ are generated by twisted orbits contained in $W_2\setminus W_1.$ In this region $\psi_m=\OP{id}$ and $(H^m,J^m)=(H^1,J^1).$ Hence 
\[CF_\ast^{<0}(W_2,\psi_m, H^m,J^m),\quad m\in\mathbb{N}\]
are generated by the same set of generators. This does not imply that 
\[HF_\ast^{<0}(W_2,\psi_m, H^m,J^m),\quad m\in\mathbb{N}\] 
are isomorphic, because the differentials may differ. However,
\[\dim HF_\ast^{<0}(W_2,\psi_m, H^m, J^m)\leqslant \dim CF_\ast^{<0}(W_2,\psi_m, H^m, J^m)= \dim CF_\ast^{<0}(W_2,\psi_1, H^1, J^1),\]
and the number on the right-hand side does not depend on $m.$ The long exact sequence \eqref{eq:lsequence} implies
\[\abs{\dim HF\left(W_2,\psi_m,\frac{\varepsilon}{2}\right)-\dim HF(W_1,\psi_m,\varepsilon)}< 2\dim CF^{<0}(W_2,\psi_1,H^1,J^1).\]
Finally, note that
\[HF\left(W_2,\psi_m,\frac{\varepsilon}{2}\right)\cong HF\left(W_2, \phi^m,\frac{\varepsilon}{2}\right)\cong HF(W_2, \phi^m,\varepsilon), \]
and 
\[HF(W_1,\psi_m,\varepsilon)\cong HF(W_1,\phi^m,\varepsilon).\]
This finishes the proof.
\end{proof}
The following proposition shows that symplectomorphisms which are supported in the cylindrical part of a Liouville domain have finite iterated ratio. This proposition will not be used in the rest of the paper. 

\begin{proposition}
Let $(W,\lambda)$ be a Liouville domain, and let $\phi:W\to W$ be an exact symplectomorphism that is compactly supported in $\partial W\times (0,1).$ Then
\[\kappa(W,\phi)<\infty.  \]
\end{proposition}
\begin{proof}
The idea of the proof is the following. We construct an exact symplectomorphism $\psi_m: W\to W$ that is symplectically isotopic to $\phi^m$ relative to the boundary, and such that it consists of $m$ ``copies'' of $\phi$ with disjoint supports. By choosing a suitable Hamiltonian, one can ensure that each ``copy'' contributes the same number of generators to the Floer chain complex. Since the number of ``copies'' grows linearly with respect to $m,$ the iterated ratio has to be smaller than the number of generators furnished by a single ``copy.''

Now, we implement the idea rigorously. By the assumption, there exists $ \delta\in \left(0,  \frac{1}{2}\right) $ such that $\phi$ is compactly supported in $\partial W\times (2\delta, 1).$ Let $H_t: \hat{W}\to\mathbb{R}$ be a Hamiltonian equal to 
\[ (x,r)\mapsto \varepsilon r \]
on $\partial W\times (\delta, 2\delta) \cup \partial W\times (1,\infty)$ such that
\[ H_{t+1}= H_t\circ \phi, \]
and such that
\[ \det\left( d(\phi \circ \psi_1^H)(x)-\OP{id} \right)\not=0 \]
for all fixed points $x$ of $\phi \circ \psi_1^H.$ The non-degeneracy condition implies that there are finitely many Hamiltonian twisted orbits. We denote by $A$ and $B$ the numbers of such orbits in $W^\delta$ and $\partial W\times (2\delta, 1),$ respectively. 

Fix $m\in \mathbb{N}.$ Denote by $\varphi _0 : W\to W$ the symplectomorphism 
\[ \varphi_0:= \psi^\lambda_T\circ \phi \circ \left( \psi^\lambda_T\right)^{-1}, \]
where $T:= (m-1)\ln \delta - m\ln (1+\delta), $ and $\psi^\lambda_t$ is the Liouville flow. Let $\varphi_j,\: j\in\{1,\ldots, m-1\}$ be the symplectomorphism given by
\[ \varphi_j:= \psi^\lambda_{j c}\circ \varphi_0\circ \left( \psi_{jc}^\lambda \right)^{-1}, \]
where $c := \ln(1+\delta)-\ln\delta.$ Similarly, denote
\begin{align*}
H^0_t&:= H_t\circ \left( \psi^\lambda_T\right)^{-1},\\
H^j_t&:= H^0_t \circ \left( \psi^\lambda_{jc}\right)^{-1},\quad j\in\{1,\ldots, m-1\}.
\end{align*}
The symplectomorphisms $\varphi_0, \varphi_1, \ldots, \varphi_{m-1}$ are all symplectically isotopic to $\phi$ relative to the boundary, and they are compactly supported in 
\begin{equation}
\label{eq:copies}
\partial W\times \left(r_0, r_0 e^{c}\right),\quad  \partial W\times \left(r_0e^{c}, e_0 e^{2c}\right),\quad \ldots,\quad \partial W\times \left(r_0e^{(m-1)c}, r_0 e^{mc} \right) 
\end{equation}
respectively. Here, 
\[ r_0:=\frac{\delta^m}{(1+\delta)^m}, \]
and $c$ is as above. (It follows that $r_0 e^{mc}=1$ and consequently the sets \eqref{eq:copies} are all subsets of $W$. Hence, $\varphi_j$ is compactly supported in the interior of $W$ for all $j\in\{0,\ldots, m-1\}.$) Let $\psi_m: W\to W$ be the exact symplectomorphism 
\[ \psi_m:= \varphi_0\circ \cdots\circ \varphi_{m-1}, \]
and let $G_t: \hat{W}\to \mathbb{R}$ be the Hamiltonian defined by
\[ G_t(p):=\left\{ \begin{matrix}
H^0_t(p) & \text{for }p\in W^{r_0},\\
H^j_t(p) &\text{for }p\in \partial W\times \left( r_0 e^{(j-1)c}, r_0 e^{jc} \right), j\in\{1,\ldots, m-1\}\\
\varepsilon r &\text{otherwise.}
\end{matrix} \right. \]
By construction, $G_t$ has exactly $A+ mB$ Hamiltonian $\psi_m$-twisted loops, and they are all non-degenerate. This implies
\[ \dim HF (W,\phi^m, \varepsilon)=\dim HF (W,\psi_m, \varepsilon)\leqslant A+mB. \]
The first equality follows from $\phi^m$ being symplectically isotopic to $\psi_m$ relative to the boundary. Hence,
\[ \kappa(W,\phi)\leqslant \lim_{m\to \infty} \frac{A+ mB}{m}= B <\infty, \] 
and the proof is finished. 
\end{proof}
\subsection{Examples}
In this section, we compute the iterated ratio for the fibered Dehn twist $\tau : D^\ast\mathbb{S}^n\to D^\ast\mathbb{S}^n$ \cite{MR1765826,BiranGiroux} of the disk cotangent bundle of the sphere $\mathbb{S}^n$ endowed with the standard Riemannian metric. The Floer homology $HF(D^\ast\mathbb{S}^n,\tau^m,\varepsilon),$ $m\in\mathbb{N},$ is isomorphic (as a $\mathbb{Z}_2$-vector space) to $HF(D^\ast\mathbb{S}^n,\OP{id}, 2\pi m+\varepsilon).$ This follows from Proposition~2.30 and Remark~3.2 in \cite{Uljarevic}. Hence
\[ \kappa(D^\ast\mathbb{S}^n,\tau)=\lim_{m\to \infty}\frac{\dim HF(D^\ast\mathbb{S}^n,\OP{id}, 2\pi m+\varepsilon)}{m}. \]

\begin{proposition}\label{pro:sphereslopes}
Let $m$ and $n$ be positive integers with $n>2,$ let $\varepsilon\in(0,2\pi),$ and let $D^\ast\mathbb{S}^n$ be the disk cotangent bundle of $\mathbb{S}^n$ with respect to the standard Riemannian metric. Then,
\[ HF_k(D^\ast\mathbb{S}^2, 2\pi m+\varepsilon)\cong \left\{\begin{matrix} \mathbb{Z}_2& \text{for }k\in\{0,1,2m+1,2m+2\},\\
\mathbb{Z}^2_2&\text{for }k\in\mathbb{N} \text{ and } 2\leqslant k\leqslant 2m,\\
0&\text{otherwise}, 
\end{matrix} \right. \]
and
\[ HF_k(D^\ast\mathbb{S}^n, 2\pi m+\varepsilon)\cong \left\{\begin{matrix} \mathbb{Z}_2& \text{for }k\in\{\ell(n-1), \ell(n-1)+n\::\:\ell\in\mathbb{Z}\:\&\:0\leqslant\ell\leqslant2m\},\\
0&\text{otherwise.}
\end{matrix} \right. \]
In particular, 
\[ \dim HF_k(D^\ast\mathbb{S}^n, 2\pi m+\varepsilon)= \dim HF_k(D^\ast\mathbb{S}^2, 2\pi m+\varepsilon)= 4m+2. \]
\end{proposition}
\begin{proof}
We will consider only the case of $\mathbb{S}^2.$ The proof for other cases is analogous and even simpler.\\ 
The Reeb flow on $S^\ast\mathbb{S}^2:=\partial D^\ast\mathbb{S}^2$ is periodic with minimal period equal to $2\pi.$ Hence, there exists a homology long exact sequence
\begin{equation}\label{eq:les}
\begin{tikzcd}[column sep=small]
\cdots \arrow{r} & HF_k(D^\ast\mathbb{S}^2,2\pi\ell+\varepsilon) \arrow{r}
& HF_k(D^\ast\mathbb{S}^2, 2\pi(\ell+1) + \varepsilon) \arrow{r}
\arrow[draw=none]{d}[name=Z, shape=coordinate]{}
& H_{k+\Delta_{\ell+1}} (S^\ast\mathbb{S}^2) \arrow[rounded corners,
to path={ -- ([xshift=2ex]\tikztostart.east)
|- (Z) [near end]\tikztonodes
-| ([xshift=-2ex]\tikztotarget.west)
-- (\tikztotarget)}]
{dll} \\
 \arrow[draw=none]{r} & HF_{k-1}(D^\ast\mathbb{S}^2,2\pi\ell+\varepsilon)\arrow{r}
& \cdots,
\end{tikzcd}
\end{equation}
where $\ell\in \mathbb{N}\cup \{0\},$ and $\Delta_{\ell}$ is a shift in grading \cite[Lemma~3.6]{Seidel2014}. The shift $\Delta_\ell$ is equal to $-(2\ell-1)$ \cite[Proposition~5.12]{Kwon&vanKoert}. The long exact sequence \eqref{eq:les} implies
\begin{equation*}
HF_k(D^\ast\mathbb{S}^2,2\pi\ell+ \varepsilon)\overset{\cong}{\longrightarrow} HF_k(D^\ast\mathbb{S}^2,2\pi(\ell+1)+ \varepsilon),
\end{equation*}
for $k<2\ell.$ Therefore
\begin{equation}\label{eq:hfsh}
HF_k(D^\ast\mathbb{S}^2,2\pi\ell +\varepsilon)\cong SH_k(D^\ast\mathbb{S}^2),\quad\text{for } k<2\ell.
\end{equation} 
We prove the proposition by induction. Using \eqref{eq:les} with $\ell=0,$ we get
\begin{align*}
&\dim HF_2(D^\ast\mathbb{S}^2,2\pi+\varepsilon )=\dim HF_3(D^\ast\mathbb{S}^2, 2\pi+ \varepsilon)+ 1,\\
& HF_4(D^\ast\mathbb{S}^2, 2\pi+ \varepsilon)\cong \mathbb{Z}_2\\
& HF_k(D^\ast\mathbb{S}^2, 2\pi+ \varepsilon)=0,\quad\text{for }k\not\in\{0,1,2,3,4\},\\
& \dim HF_3(D^\ast\mathbb{S}^2, 2\pi+ \varepsilon)\leqslant 1.
\end{align*}
Consequently,
\begin{equation}\label{eq:ineq}
\dim HF_2(D^\ast\mathbb{S}^2, 2\pi+ \varepsilon)\leqslant 2. 
\end{equation}
The equation \eqref{eq:hfsh} implies
\begin{align*}
& HF_0(D^\ast\mathbb{S}^2,2\pi +\varepsilon)\cong SH_0(D^\ast\mathbb{S}^2)\cong \mathbb{Z}_2,\\
& HF_1(D^\ast\mathbb{S}^2,2\pi +\varepsilon)\cong SH_1(D^\ast\mathbb{S}^2)\cong\mathbb{Z}_2,\\
& HF_2(D^\ast\mathbb{S}^2,4\pi +\varepsilon)\cong SH_2(D^\ast\mathbb{S}^2)\cong\mathbb{Z}_2^2. 
\end{align*}
(The computation of the symplectic homology for $D^\ast\mathbb{S}^2$ can be done by using \cite{Viterbo} and \cite[page 21]{MR0649625}.)
Consider the segment
\begin{center}
\begin{tikzcd}
\cdots\arrow{r}& HF_2(D^\ast\mathbb{S}^2,2\pi+\varepsilon) \arrow{r}& HF_2(D^\ast\mathbb{S}^2,4\pi+\varepsilon)\arrow{r} &0
\end{tikzcd} 
\end{center}
of the long exact sequence \eqref{eq:les} with $\ell=1.$ It implies 
\[ \dim HF_2(D^\ast\mathbb{S}^2,2\pi +\varepsilon)\geqslant\dim HF_2(D^\ast\mathbb{S}^2,4\pi +\varepsilon)=2. \]
Therefore (see \eqref{eq:ineq})
\[ HF_2(D^\ast\mathbb{S}^2,2\pi +\varepsilon)\cong \mathbb{Z}_2^2\quad\text{and}\quad HF_3(D^\ast\mathbb{S}^2,2\pi +\varepsilon)\cong \mathbb{Z}_2 \]
This proves the basis of the induction. Assume now the claim holds for $1,\ldots, m.$ The groups $HF_k(D^\ast\mathbb{S}^2,2\pi(m+1)+\varepsilon)$ for $k\in\{0,\ldots, 2m-1\}\cup\{2m+4\}$ can be computed directly from \eqref{eq:les} with $\ell=m.$ The isomorphism $\eqref{eq:hfsh}$ computes the groups $HF_{2m}(\cdots)$ and $HF_{2m+1}(\cdots).$ Finally, the groups $HF_{2m+2}(\cdots)$ and $HF_{2m+3}(\cdots)$ are computed in the same way as $HF_2(D^\ast\mathbb{S}^2,2\pi+ \varepsilon)$ and $HF_3(D^\ast\mathbb{S}^2,2\pi+ \varepsilon)$ in the basis of the induction. (We are not going to repeat the argument here.) This finishes the proof.
\end{proof}

\begin{proof}[Proof of Theorem~\ref{thm:taukappa}]
By the Weinstein neighbourhood theorem, there exists a neighbourhood $V$ of $L$ in $W$ that is symplectomorphic to the disk-cotangent bundle $D^\ast_\rho\mathbb{S}^n$ for small enough radius $\rho\in(0,\infty).$ The manifold $V$ is not, in general, a Liouville domain  when considered with the 1-form $\lambda.$ Namely, the Liouville vector field $X_\lambda$ might not be pointing out on the boundary $\partial V.$ There is, however, a 1-form $\lambda_0$ furnished by the canonical Liouville form on $D^\ast_\rho\mathbb{S}^n$ that makes $V$ into a Liouville domain. The restriction $\left.(\lambda-\lambda_0)\right|_{L}= \left.\lambda\right|_L$ is an exact form because $L$ is an exact Lagrangian submanifold. This, together with the inclusion $L\hookrightarrow V$ being homotopy equivalence, implies that there exists a function $f:V\to \mathbb{R}$ such that $\lambda-\lambda_0= df.$ Let $\chi: V\to [0,1]$ be a function equal to 0 on $V^{\frac{1}{3}}$ and to 1 on the complement of $V ^{\frac{2}{3}}.$ Consider the Liouville form
\[ \lambda_1= \lambda_0 + d(\chi f). \] 
Note that $\lambda_1$ is equal to $\lambda_0$ on $V^{\frac{1}{3}}$ and can be extended by $\lambda$ on the complement of $V.$ Without loss of generality, assume that $\tau_L$ is compactly supported in the interior of $V^{\frac{1}{4}}.$  Theorem~\ref{thm:ambient} implies
\begin{equation}\label{eq:vlambda1W}
\kappa(V^{\frac{1}{4}}, \lambda_1,\tau_L^2)= \kappa(W,\lambda_1,\tau_L^2).
\end{equation} 
The Liouville domain $(V^{\frac{1}{4}}, \lambda_1)$ is isomorphic to the disk-cotangent bundle $D^\ast_\varepsilon \mathbb{S}^n$ for some radius $\varepsilon$ because $\lambda_1$ is equal to $\lambda_0$ on $V^{\frac{1}{4}}.$ Besides, the symplectomorphism $\tau^2_L$ is symplectically isotopic to the fibered Dehn twist. Hence, due to Proposition~\ref{pro:sphereslopes},
\begin{equation}\label{eq:vlambda1=4} \kappa(V^{\frac{1}{4}}, \lambda_1,\tau_L^2)= 4. 
\end{equation}
Finally, the 1-form $\lambda-\lambda_1$ is, by construction, the exterior derivative of a function $W\to\mathbb{R}$ that is equal to 0 near the boundary. Hence Lemma~\ref{lem:changinglambda}, \eqref{eq:vlambda1W}, and \eqref{eq:vlambda1=4} finish the proof. 
\end{proof}

\begin{proof}[Proof of Corollary~\ref{cor:tauinf}]
The corollary follows form Theorem~\ref{thm:ambient} and Theorem~\ref{thm:taukappa}.
\end{proof}
\subsection{An application to closed geodesics}

\begin{definition}
The \textbf{visible rank} $r(W,a)$ of a Liouville domain $W$ and an admissible slope $a$ is defined to be
\[r(W,a):=\dim_{\mathbb{Z}_2}\iota\left(HF(W,\OP{id},a)\right),\]
where 
\[\iota\::\: HF_\ast(W,\OP{id},a)\to HF_\ast(W,\OP{id},\infty)=SH_\ast(W;\mathbb{Z}_2)\]
is the natural morphism.
\end{definition}

\begin{lemma}\label{lem:transferiso}
Let $W_1$, $W_2$ be as in Section~\ref{sec:notation}. Assume there exists $r\in(0,1)$ such that $W_2^r\subset W_1.$ Then, the transfer morphism
\begin{equation}\label{eq:trensferiso}
SH(W_2)=HF(W_2,\infty)\to HF(W_1,\infty)=SH(W_1) 
\end{equation}
is an isomorphism. 
\end{lemma}
\begin{proof}
The map
\begin{equation}
\hat{W}_1\to \hat{W}_2\::\:
p\mapsto\left\{\begin{matrix} j (p)&\text{for }p\in W_1\\ (\OP{pr}_1(j(p)),\OP{pr}_2(j(x))\cdot r)&\text{for }p=(x,r)\in\partial W_1\times(0,\infty) \end{matrix} \right.
\end{equation}
is a diffeomorphism that respects the Liouville forms. Here $j:W_1\hookrightarrow W_2$ stands for the inclusion, and $\OP{pr}_1, \OP{pr}_2$ are the projections $\partial W_2\times(0,\infty)\to \partial W_2$ and $\partial W_2\times(0,\infty)\to (0,\infty),$ respectively. We will identify $\hat{W}_1$ and $\hat{W}_2$ via this map. By the assumption, there exists $s\in(0,\infty)$ such that $W_2\subset W_1^s.$ The transfer morphisms
\begin{align*}
SH_\ast(W_1^s)&\to SH_\ast(W_1),\\
SH_\ast(W_2)&\to SH_\ast(W_2^r)
\end{align*} 
are isomorphisms due to Lemma~4.16 in \cite{Gutt2015}. They fit (because of the functoriality of the transfer morphisms \cite{Gutt2015}) into the commutative diagram
\begin{center}
\begin{tikzcd}
SH_\ast(W_1^s)\arrow{r}\arrow[bend right]{r d}&SH_\ast(W_2)\arrow{d}\arrow[bend left]{r d}\\
&SH_\ast(W_1)\arrow{r}& SH_\ast(W_2^r)
\end{tikzcd}
\end{center}
Hence \eqref{eq:trensferiso} is an isomorphism as well.
\end{proof}

\begin{theorem}\label{thm:visiblerankcomparison}
Let $W_1,W_2$ be as in Section~\ref{sec:notation}. Assume there exists $r\in(0,1)$ such that $W_2^r\subset W_1.$ Then,
\[ r(W_2,a)\leqslant r(W_1,a) \]
for all slopes $a\in(0,\infty)$ that are admissible with respect to both $W_1$ and $W_2.$ 
\end{theorem}
\begin{proof}
Since $\partial W_2$ is compact, the set of admissible slopes with respect to $W_2$ is open (see Lemma~\ref{lem:periodsclosed}). Let $\varepsilon>0$ be such that the elements of $[a-\varepsilon,a]$ are all admissible with respect to $W_2.$ Lemma~2.25 in \cite{Uljarevic} implies
\[ r(W_2,a)=r(W_2,a-\varepsilon). \]
Consider the commutative diagram
\begin{center}
\begin{tikzcd}
SH_\ast(W_2)\arrow{r} &SH_\ast(W_1)\\
HF_{\ast}(W_2,a-\varepsilon)\arrow{r}\arrow{u} &H_{\ast}(W_1,a)\arrow{u},
\end{tikzcd}
\end{center}
whose existence follows from Theorem~\ref{thm:transfermorphism}. The upper horizontal arrow of the diagram is an isomorphism due to Lemma~\ref{lem:transferiso}. Hence
\[  r(W_2,a)=r(W_2,a-\varepsilon)\leqslant r(W_1,a).\]
\end{proof}

\begin{corollary}\label{cor:geogeneral}
Let $Q$ be a closed manifold, and let $g_0$ and $g_1$ be two Riemannian metrics on $Q$ such that $g_1\leqslant g_0,$ i.e. $g_1(v,v)\leqslant g_0(v,v)$ for all $v\in TQ.$ Assume
\begin{equation}\label{eq:ineqgeogenera}
r(D^\ast Q,a)>\dim_{\mathbb{Z}_2} H(Q,\mathbb{Z}_2),
\end{equation}
where $D^\ast Q$ is the disk-cotangent bundle of $Q$ with respect to $g_0$ and $a$ is a positive real number that is not a period of some closed $g_0$-geodesic on $Q.$ Then there exists a closed $g_1$-geodesic of length less than $a.$
\end{corollary}
\begin{proof}
Denote $D^\ast Q$ by $W_2$ and the disk-cotangent bundle of $Q$ with respect to $g_1$ by $W_1.$ The condition $g_1\leqslant g_0$ implies (see Lemma~\ref{lem:twometrics} below) $W_1\subset W_2.$

We argue by contradiction, i.e. assume there exists no (non-constant) closed $g_1$-geodesic on $Q$ with the length less than $a.$ Then, the numbers in $(0,a]$ are all admissible with respect to $W_1.$ Hence
\[ r(W_1,a)=r(W_1,\varepsilon)\leqslant \dim HF(W_1,\varepsilon)=\dim H(W_1;\mathbb{Z}_2)=\dim H(Q;\mathbb{Z}_2), \]
where $\varepsilon>0$ is small enough. This, together with \eqref{eq:ineqgeogenera}, implies
\[ r(W_1,a)<r(W_2,a), \]
a relation in contradiction with Theorem~\ref{thm:visiblerankcomparison}.   
\end{proof}

\begin{example}
Let $(\mathbb{S}^n,g_0)$ be the $n$-dimensional sphere with the standard Riemannian metric, and let $\varepsilon\in(0,2\pi).$ The visible rank $r(D^\ast\mathbb{S}^n,2\pi+ \varepsilon)$ is equal to 6 and, therefore, greater than the sum of the Betty numbers of $\mathbb{S}^n.$ Hence any Riemannian metric $g$ on $\mathbb{S}^n$ with $g\leqslant g_0$ has a non-constant closed geodesic of length less than or equal to $2\pi+\varepsilon.$ Since $\varepsilon$ is an arbitrary number in $(0,2\pi),$ $g$ has a non-constant closed geodesic of length less than or equal to $2\pi.$ 
\end{example}

\begin{lemma}\label{lem:twometrics}
Let $Q$ be a manifold, and let $g_0,g_1$ be two Riemann metrics on $Q$ such that $g_1\leqslant g_0.$ Than,
\[ D^\ast Q(g_1)\subset D^\ast Q(g_0),  \]
where $D^\ast Q(g_i),\,i=0,1$ stand for the disk cotangent bundle of $Q$ with respect to $g_i,\,i=0,1.$
\end{lemma}
\begin{proof}
It is enough to prove that, for every $\alpha\in T^\ast Q,$ the inequality
\[ \abs{v}_0\leqslant\abs{w}_1 \]
holds, where $\abs{\cdot}_i$ is the norm induced by $g_i,\,i=0,1,$ and $v,w$ are unique vectors in $TQ$ such that
\[ g_0(v,\cdot)=\alpha=g_1(w,\cdot). \] 
Using the polarization identity
\[ 2g_1(v,w)= \abs{v}_1+\abs{w}_1-\abs{v-w}_1 \]
and the inequality $\abs{v}_0\geqslant \abs{v}_1,$ we get
\[\begin{split}
\abs{w}^2_1-\abs{v}^2_0 &= \abs{w}^2_1+\abs{v}^2_1-\abs{w-v}^2_1 -\abs{v}^2_1+ \abs{w-v}^2_1- \abs{v}^2_0\\
&=2g_1(w,v)-\abs{v}^2_1+ \abs{w-v}^2_1- \abs{v}^2_0\\
&=2\alpha(v)-\abs{v}^2_1+ \abs{w-v}^2_1- \abs{v}^2_0\\
&=2\abs{v}^2_0-\abs{v}^2_1+ \abs{w-v}^2_1- \abs{v}^2_0\\
&=\abs{v}_0^2-\abs{v}_1^2+ \abs{w-v}_1^2\geqslant 0.
\end{split}\]
\end{proof}

\appendix
\section{Technical lemmas}

\begin{lemma}\label{lem:tehlem}
Let $r_0,\beta_0, \beta_1$ be real numbers and let $\ell,\alpha$ be positive real numbers. Then, the following conditions are equivalent.
\begin{itemize}
\item There exists a convex, strictly increasing function
\[h:[r_0,r_0+\ell]\to\mathbb{R}\]
such that the function
\[\mathbb{R}\to\mathbb{R}:r\mapsto\left\{ \begin{matrix}
\beta_0& r\leqslant r_0,\\
h(r)& r\in[r_0,r_0+\ell],\\
\alpha r+\beta_1& r\geqslant r_0+\ell
\end{matrix} \right.\]
is $C^\infty.$ 
\item The following inequalities hold
\begin{equation}\label{eq:tehlem1}
\begin{split}
&\alpha r_0<\beta_0-\beta_1,\\
&\beta_1-\beta_0+\alpha r_0+\alpha \ell>0.
\end{split}
\end{equation}
\end{itemize} 
\end{lemma}
\begin{proof}
We first assume there exists such a function and prove \eqref{eq:tehlem1}. Since $h$ is increasing, we get
\[\beta_0=h(r_0)<h(r_0+\ell)=\beta_1+\alpha r_0+\alpha\ell,\]
or equivalently
\[\beta_1-\beta_0+\alpha r_0+\alpha \ell>0.\]
The Newton-Leibniz formula, together with $h$ being convex, implies
\[\alpha r_0+\alpha\ell+\beta_1-\beta_0= h(r_0+\ell)-h(r_0)=\int_{r_0}^{r_0+\ell}h'(r)dr<\int_{r_0}^{r_0+\ell}\alpha dr=\alpha\ell.\]
Hence \eqref{eq:tehlem1}. (The function $h'$ is continuous and not constant. Therefore the strict inequality holds.)

Now, we construct $h$ under assumption~\eqref{eq:tehlem1}. Consider the famous cut-off function
\[\mathbb{R}\to \mathbb{R}\::\:r\mapsto\left\{ \begin{matrix} e^{-\frac{1}{1-r^2}}& \abs{r}<1\\ 0&\text{otherwise.}\end{matrix} \right.\]
Denote by $f_2:\mathbb{R}\to\mathbb{R}$ the function given by
\[f_2(r):=\frac{\alpha}{c\ell}f_1\left(\frac{2r-2r_0-\ell}{\ell}\right),\]
where 
\[c:=\int_\mathbb{R}f_1(r)dr.\]
It is bigger than 0 on $(r_0,r_0+\ell)$ and equal to 0 elsewhere. In addition $\int_{\mathbb{R}}f_2(r)dr=\alpha.$ Hence the function
\[f_3:\mathbb{R}\to\mathbb{R}\::\: r\mapsto \int_{r_0}^rf_2(t)dt\]
is equal to 0 on $(-\infty,r_0],$ strictly increasing on $(r_0,r_0+\ell),$ and equal to $\alpha$ on $[r_0+\ell,\infty).$

If 
\[\beta_1-\beta_0+r_0\alpha+\ell\alpha=\int_{r_0}^{r_0+\ell} f_3(r)dr=:c_1,\]
then the function 
\[h(r):=\beta_0+\int_{r_0}^r f_3(t)dt\]
satisfies the conditions. Assume $\beta_1-\beta_0+r_0\alpha+\ell\alpha>c_1.$ Let $k_1\geqslant 1$ be a real number such that 
\[\int_{r_0}^{r_0+\ell}f_4(r)dr= \beta_1-\beta_0+r_0\alpha+\ell\alpha,\]
where 
\[f_4(r):=f_3(k_1r-k_1r_0+r_0).\]
Such a number exists because
\[g_1:t\mapsto \int_{r_0}^{r_0+\ell}f_3(tr-tr_0+r_0)dr\]
is a continuous function satisfying
\begin{align*}
&g_1(1)=c_1< \beta_1-\beta_0+r_0\alpha+\ell\alpha,\\
&\lim_{t\to\infty}g(t)=\ell\alpha> \beta_1-\beta_0+r_0\alpha+\ell\alpha.
\end{align*}
The last inequality is due to \eqref{eq:tehlem1}. In this case, we can take $h$ to be
\[h(r):=\beta_0+\int_{r_0}^rf_4(t)dt.\]
Assume now $\beta_1-\beta_0+r_0\alpha+\ell\alpha<c_1.$ Let $\varepsilon\in(0,1]$ and $k_2\in[1,\infty)$ be real numbers such that
\[\int_{r_0}^{r_0+\ell}f_3(r)f_5(r)dr=\beta_1-\beta_0+r_0\alpha+\ell\alpha,\]
where
\[f_5(r):=\varepsilon+ \frac{1-\varepsilon}{\alpha} f_3(k_2r-k_2r_0-k_2\ell+r_0+\ell).\]
The function
\[h(r):=\beta_0+\int_{r_0}^{r_0+\ell}f_3(r)f_5(r)dr\]
satisfies the conditions. In what follows, we show that the numbers $\varepsilon$ and $k_2$ exist. Assume the contrary. Since the function
\[g_2:(0,1]\times[1,\infty:(s,t)\mapsto \int_{r_0}^{r_0+\ell} f_3(r)\left(s+ \frac{1-s}{\alpha} f_3(tr-tr_0-t\ell+r_0+\ell)\right)dr \]
is continuous, $g_2(1,t)=c_1>\beta_1-\beta_0+r_0\alpha+\ell \alpha,$ and the space $(0,1]\times[1,\infty)$ is connected, the inequality
\[g_2(s,t)>\beta_1-\beta_0+r_0\alpha+\ell \alpha\]
holds for all $(s,t)\in (0,1]\times[1,\infty).$ The function $r\mapsto f_3(tr-tr_0-t\ell+r_0+\ell)$
is equal to 0 on $\left(-\infty,r_0+\frac{t-1}{t}\ell \right),$ hence
\[g_2(s,t)\leqslant s c_1 + \frac{\ell}{t} \max f_3\left(s+\frac{\max f_3}{\alpha}\right).\]
Therefore, for $s$ small enough and $t$ large enough, we get
\[g_2(s,t)<\beta_1-\beta_0+r_0\alpha+\ell\alpha\]
(the number on the right-hand side is positive due to \eqref{eq:tehlem1}). This is a contradiction and the proof is finished.
\end{proof}

\begin{lemma}\label{lem:tehlem2}
Let $r_0,\beta_0, \beta_1$ be real numbers and let $\ell,\alpha$ be positive real numbers. Then, the following conditions are equivalent.
\begin{itemize}
\item There exists a concave, strictly increasing function
\[h:[r_0,r_0+\ell]\to\mathbb{R}\]
such that the function
\[\mathbb{R}\to\mathbb{R}:r\mapsto\left\{ \begin{matrix}
\alpha r+\beta_0& r\leqslant r_0,\\
h(r)& r\in[r_0,r_0+\ell],\\
\beta_1& r\geqslant r_0+\ell
\end{matrix} \right.\]
is $C^\infty.$ 
\item The following inequalities hold
\begin{equation}
0<\beta_1-\beta_0-\alpha r_0<\ell \alpha.
\end{equation} 
\end{itemize}
Moreover, given $\alpha_1\in\mathbb{R},$ the inequality 
\[h(r)>\alpha_1 r\]
holds for all $r\in [r_0,r_0+\ell]$ if, and only if, 
\[\alpha r_0+\beta_0>r_0\alpha_1\quad\text{and}\quad \beta_1> \alpha_1(r_0+\ell).\]
\end{lemma}
\begin{proof}
The function $h:[r_0,r_0+\ell]\to\mathbb{R}$ is concave, strictly increasing if, and only if, the function
\[[-r_0-\ell,-r_0]\to\mathbb{R}\::\: r\mapsto -h(-r)\]
is convex and strictly increasing. Hence Lemma~\ref{lem:tehlem} implies the first part of the lemma. The second part follows from the concavity of the function $h.$
\end{proof}

\begin{lemma}\label{lem:periodsclosed}
Let $M$ be a closed contact manifold with contact form $\alpha.$ Then, the set $S(M,\alpha)$ of all periods of periodic Reeb orbits on $(M,\alpha)$ is a closed subset of $\mathbb{R}.$
\end{lemma}
\begin{proof}
Let $d:M\times M\to[0,\infty)$ be a metric on $M.$ Denote by $\sigma_t:M\to M$ the Reeb flow of $(M,\alpha).$ Consider the function
\[ f\::\:M\times\mathbb{R}\to [0,\infty) \::\:(x,t)\mapsto d(x,\sigma_t(x)).\]
The set $f^{-1}(\{0\})$ is a closed subset of $M\times\mathbb{R},$ because $f$ is a continuous map. Since $M$ is compact, the projection
\[ \OP{pr}_2\::\:M\times\mathbb{R} \to \mathbb{R} \]
is a closed map. Hence
\[ \OP{pr}_2\left( f^{-1}(\{0\}) \right)=S(M,\alpha) \]
is a closed subset of $\mathbb{R}.$
\end{proof}
\printbibliography
\end{document}

%% file: transferfigure.tex
\begin{tikzpicture}[scale=0.9, every edge quotes/.style={fill=white,font=\footnotesize}]
\draw [rounded corners] (0,0) to [out=90, in=180] (2,1.5) to [out=0, in=180](5,0.75)--(16,2.675);
\draw [rounded corners] (0,0) to [out=-90, in=180] (1.5,-1.5) to [out=0, in=180] (5,-0.75)--(16,-2.675);

\draw (5,0.75) arc [start angle= 90, end angle=270, x radius= 0.14, y radius= 0.75];
\draw [dashed] (5,-0.75) arc [start angle= -90, end angle=90, x radius= 0.14, y radius= 0.75];
\draw (8,1.275) arc [start angle= 90, end angle=270, x radius= 0.238, y radius= 1.275];
\draw [dashed] (8,-1.275) arc [start angle= -90, end angle=90, x radius= 0.238, y radius= 1.275];
\draw (13,2.15)arc [start angle= 90, end angle=270, x radius= 0.402, y radius= 2.15];
\draw [dashed] (13,-2.15) arc [start angle= -90, end angle=90, x radius= 0.402, y radius= 2.15];
\draw (16,0) ellipse [x radius = 0.5, y radius=2.675];

\draw (1, 0.2) to [out=-45, in=210] (3,0.2);
\draw (1.6, -0.1) to [out= 30, in=135] (2.4, 0);
\draw [shift={(8,0)}](1, 0.2) to [out=-45, in=210] (3,0.2);
\draw [shift={(8,0)}](1.6, -0.1) to [out= 30, in=135] (2.4, 0);

\draw [dotted,->] (0,-7.15)--(0,7.1);
\draw [dotted] (5,-0.75)--(5,7.1);
\draw [dotted] (5.5,-2.25)--(5.5,7.1);
\draw [dotted] (6,-2.95)--(6,7.1);
\draw [dotted] (7.5,-3.65)--(7.5,7.1);
\draw [dotted] (8,-4.35)--(8,7.1);
\draw [dotted] (13,-5.05)--(13,7.1);
\draw [dotted] (13.5,-5.75)--(13.5,7.1);
\draw [dotted] (14,-6.45)-- (14,7.1);

\draw [dotted] (0,6)--(8,6);
\draw [dotted, <->] (0,-2)--(5.5,-2) node [midway, fill=white] {$W_1^{\delta_1}$};
\draw [dotted, <->] (0,-2.7)--(6,-2.7) node [midway, fill=white] {$W_1^{2\delta_1}$};
\draw [dotted, <->] (0,-3.4)--(7.5,-3.4) node [midway, fill=white] {$W_1^{1-\delta_2}$};
\draw [dotted, <->] (0,-4.1)--(8,-4.1) node [midway, fill=white] {$W_1$};
\draw [dotted, <->] (0,-4.8)--(13,-4.8) node [midway, fill=white] {$W_2$};
\draw [dotted, <->] (0,-5.5)--(13.5,-5.5) node [midway, fill=white] {$W_2^{\delta_3}$};
\draw [dotted, <->] (0,-6.2)--(14,-6.2) node [midway, fill=white] {$W_2^{2\delta_3}$};
\draw [dotted, <->>] (0,-6.9)--(16,-6.9) node [midway, fill=white] {$\hat{W}_2$};

\draw [rounded corners, thick](0,3)--(5.6,3)--(7.9,6)--(13.6,6)--(16,7);

\node[above right] at (0,3) {$-A$};
\node[above right] at (0,6) {$B$};
\node[below left] at (15.5,6.5) {$H$};
\end{tikzpicture}